\documentclass[12pt]{article}

\usepackage{amsmath,amssymb,amsthm}
\usepackage{graphicx}

\normalsize

\newtheorem{theorem}{Theorem}[section]
\newtheorem{lemma}[theorem]{Lemma}

\newtheorem{corollary}[theorem]{Corollary}
\numberwithin{equation}{section}

\def\kvec{{\boldsymbol{k}}}
\def\kmax{k_{\mathrm{max}}}

\def\Hrk{{\mathcal{H}_r(\kvec)}}

\def\nfrac#1#2{{\textstyle\frac{#1}{#2}}}
\def\dfrac#1#2{\lower0.15ex\hbox{\large$\frac{#1}{#2}$}}

\title{Asymptotic enumeration of sparse uniform hypergraphs with given degrees}

\author{
Vladimir Blinovsky\thanks{Supported by FAPESP (2012/13341-8, 2013/07699-0) and NUMEC/USP (Project MaCLinC/USP).}\\
\small Instituto de Matem{\' a}tica e Estat{\' \i}stica\\[-0.8ex]
\small Universidade de S{\~ a}o Paulo, 05508-090, Brazil\\[-0.8ex]
\small Institute for Information Transmission Problems\\[-0.8ex]
\small Russian Academy of Sciences\\[-0.8ex]
\small Moscow 127994, Russia\\[-0.8ex]
\small \tt vblinovs@yandex.ru\\
\and
Catherine Greenhill\thanks{Supported by the Australian Research Council grants DP120100197 and DP140101519.} \\
\small School of Mathematics and Statistics\\[-0.8ex]
\small The University of New South Wales\\[-0.8ex]
\small Sydney NSW 2052, Australia\\[-0.8ex]
\small \tt c.greenhill@unsw.edu.au\\
}

\date{9 June 2015}

\begin{document}

\maketitle
\begin{abstract}
Let $r\geq 2$ be a fixed integer.
For infinitely many $n$, let
$\kvec = (k_1,\ldots, k_n)$ be a vector of nonnegative
integers such that their sum $M$ is divisible by $r$.
We present an asymptotic enumeration formula for 
simple $r$-uniform hypergraphs with degree sequence $\kvec$.   
(Here ``simple'' means that all edges are distinct
and no edge contains a repeated vertex.) 
Our formula holds whenever the maximum degree $\kmax$ satisfies
$\kmax^{3} = o(M)$. 
\end{abstract}

\section{Introduction}\label{s:introduction}

Hypergraphs are combinatorial structures which can model
very general relational systems, including some real-world 
networks~\cite{ER,GZCN,Dih}.
Formally, a  \emph{hypergraph} or a set system is defined as a pair
$(V,E)$, where $V$ is a finite set and 
$E$ is a multiset of multisubsets of $V$. 
(We refer to elements of $E$ as \emph{edges}.)
Note that under this definition, a hypergraph may contain repeated
edges and an edge may contain repeated vertices.

If a vertex $v$ has multiplicity at least 2 in the edge $e$,
we say that $v$ is a \emph{loop} in $e$. 
A hypergraph is 
\emph{simple} if it has no loops and no repeated edges.  Here
it is possible that distinct edges may have more than one vertex
in common.  
Let $r\geq 2$ be a fixed integer.
We say that the hypergraph $(V,E)$ is $r$-\emph{uniform} 
if each edge $e\in E$ contains exactly $r$ vertices (counting multiplicities).
Uniform hypergraphs are a particular focus of study, not least
because a 2-uniform hypergraph is precisely a graph.  
We seek an 
asymptotic enumeration formula for the number of $r$-uniform simple
hypergraphs with a given degree sequence, when $r\geq 3$ is constant
and the maximum degree is not too large (the sparse range).

To state our result precisely, we need some definitions.
Let $k_{i,n}$ be a nonnegative integer for all pairs $(i,n)$ of
integers which satisfy $1\leq i\leq n$.  Then for each $n\geq 1$,
let $\kvec = \kvec(n) = (k_{1,n},\ldots, k_{n,n})$.
We usually write $k_i$ instead of $k_{i,n}$.
Define $M = \sum_{i=1}^n k_i$. 
We assume that $M$ is divisible by $r$
for an infinite number of values of $n$,
and tacitly restrict ourselves to such $n$.

We write $(a)_m$ to denote the falling factorial $a(a-1)\cdots (a-m+1)$,
for integers $a$ and $m$.
For each positive integer $t$, let $M_t = \sum_{i=1}^n (k_i)_t$.
Notice that $M_1=M$ and that $M_t\leq \kmax M_{t-1}$ for
$t\geq 2$.

Let $\Hrk$ 
be the set of simple $r$-uniform hypergraphs on the vertex
set $\{ 1,2,\ldots, n\}$ with degrees given by $\kvec = (k_1,\ldots, k_n)$.
Our main theorem is the following.

\begin{theorem}
Let $r\geq 3$ be a fixed integer.
Suppose that $n\to\infty$, $M\to\infty$ and that $\kmax$ satisfies
$\kmax\geq 2$ and $\kmax^{3} = o(M)$.  Then
\[
|\Hrk| = 
   \frac{M!}{\left(M/r\right)!\, (r!)^{M/r}\, \prod_{i=1}^{n}\, k_i!}\,\,
  \exp\biggl( - \frac{ (r-1)\, M_2}{2M} + O(\kmax^{3}/M)\, \biggr).\]
\label{main}
\end{theorem}

As a corollary, we immediately obtain the corresponding formula for 
regular hypergraphs.  Let $\mathcal{H}_r(k,n)$ denote the set of
all $k$-regular $r$-uniform hypergraphs on the vertex set $\{ 1,\ldots, n\}$,
where $k\geq 2$ is an integer, which may be a function of $n$.

\begin{corollary}
\label{main-regular}
Suppose that $n\to\infty$ and that $k$ satisfies $k\geq 2$ and
$k^{2} = o(n)$.  Then
\[ |\mathcal{H}_{r}(k,n)| = \frac{(kn)!}{(kn/r)!\, (r!)^{kn/r}\, (k!)^n}\,
   \exp\biggl( - \dfrac{1}{2}\, (k-1)(r-1) + O(k^{2}/n)\, \biggr).
\]
\end{corollary}

\subsection{History}\label{s:history}

In the case of graphs, the best asymptotic formula in the sparse
range is given by McKay and Wormald~\cite{McKW91}.  See that paper 
for further history of the problem.  
Note that their formula has a similar form to ours, but with
many more term in the exponential factor.  This is due to
the fact that it is harder to avoid creating a repeated edge
with a switching when $r=2$.  

The dense range for $r=2$ was
treated in~\cite{ranX,MW90}, but there is a gap between these
two ranges in which nothing is known.

An early result in the asymptotic enumeration of hypergraphs
was given by Cooper et al.~\cite{CFMR}, who considered simple
$k$-regular hypergraphs when $k=O(1)$.  
Dudek et al.~\cite{DFRS} proved an asymptotic formula for 
the number of
simple $k$-regular hypergraphs graphs with $k=o(n^{1/2})$.
A restatement of their result in our notation is the following:

\begin{theorem} \emph{(~\cite[Theorem 1]{DFRS})}\
For each integer $r\geq 3$, define 
\[ \kappa = \kappa(r) = \begin{cases} 1 & \text{ if $r\geq 4$,}\\
                                     \nfrac{1}{2} & \text{ if $r=3$.}
        \end{cases}
\]
Let $\mathcal{H}(r,k)$ denote the set of all simple
$k$-regular $r$-uniform hypergraphs on the vertex set $\{ 1,\ldots, n\}$.
For every $r\geq 3$, if $k=o(n^{\kappa})$ then
\[ |\mathcal{H}(r,k)| = \frac{(kn)!}{(kn/r)!\, (r!)^{kn/r}\, (k!)^n}\,
   \exp\left( -\dfrac{1}{2} (k-1)(r-1)\bigl(1 + O(\delta(n))\bigr) \right)\]
where $\delta(n) = (kn)^{-1/2} + k/n$.
\end{theorem}
Note that the factor outside the exponential part matches ours (see
Corollary~\ref{main-regular}), 
and that the exponential part of their formula can be 
rewritten as
\[ \exp\left( -\dfrac{1}{2} (k-1)(r-1) + O(k\delta(n))\right)\]
with relative error 
\[ O(k\delta(n)) = O\bigl(\sqrt{k/n} + k^2/n\bigr).\]
This relative error is only $o(1)$ when $k^2=o(n)$,
matching the range of $k$ covered by Corollary~\ref{main-regular}.
Hence Theorem~\ref{main} can be seen as an extension of~\cite{DFRS}
to irregular degree sequences.

For an asymptotic formula for the number of dense simple $r$-uniform
hypergraphs with a given degree sequence, see~\cite{KLP}.

\subsection{The model, some early results and a plan of the proof}

We work in a generalisation of the configuration model.
Let $B_1, B_2,\ldots, B_n$ be disjoint sets, which we call \emph{cells},
and define $\mathcal{B} = \bigcup_{i=0}^n B_i$. 
Elements of $\mathcal{B}$ are called points.  Assume that
cell $B_i$ contains exactly $k_i$ points, for $i=1,\ldots, n$.
We assume that there is a fixed ordering on the $M$ points of $\mathcal{B}$.

Denote by $\Lambda_r(\kvec)$ the set of all unordered partitions 
$Q = \{ U_1,\ldots, U_{M/r}\}$ of 
$\mathcal{B}$ into $M/r$ parts, where each part has exactly $r$ points.
Then
\begin{equation}
\label{Lambda}
 |\Lambda_r(\kvec)| = \frac{M!}{(M/r)!\, (r!)^{M/r}}.
\end{equation}

Each partition $Q\in\Lambda_r(\kvec)$ defines a
hypergraph $G(Q)$ on the vertex set $\{ 1,\ldots, n\}$ in a natural
way:
vertex $i$ corresponds to the cell $B_i$, and each part $U\in Q$
gives rise to an edge $e_U$ such that the multiplicity of
vertex $i$ in $e_U$ equals $|U\cap B_i|$, for $i=1,\ldots, n$.
Then $G(Q)$ is an $r$-uniform hypergraph with degree sequence $\kvec$.
The partition $Q\in\Lambda_r(\kvec)$ is called \emph{simple}
if $G(Q)$ is simple.

The edge $e_U$ has a loop at $i$ if and only if $|U\cap B_i| \geq 2$.
In this case, each pair of distinct points in $U\cap B_i$ is called
a \emph{loop} in $U$.  
We reserve the letters $e, f$ for edges in a hypergraph,
and use $U$, $W$ for parts in a partition $Q$ (that is,
in the configuration model).

Now we will consider random partitions. 
Each hypergraph in $\Hrk$ corresponds to exactly
\[ \prod_{i=1}^n k_i!\]
partitions $Q\in\Lambda_r(\kvec)$.
Hence, when $Q\in\Lambda_r(\kvec)$ is chosen uniformly at random, 
conditioned on $G(Q)$ being simple, the probability distribution of
$G(Q)$ is uniform over $\Hrk$.
Let $P_r(\kvec)$ denote the probability that a partition 
$Q\in \Lambda_r(\kvec)$ chosen uniformly at random is simple.
Then
\begin{equation}
\label{pr-equation}
  |\Hrk| = \frac{M!}{(M/r)!\, (r!)^{M/r}\, \prod_{i=1}^n k_i!}\, 
           P_r(\kvec).
\end{equation}
Hence it suffices to show that
$P_r(\kvec)$ equals the exponential factor in the
statement of Theorem~\ref{main}. 
As a first step, we identify several events which have probability
$O(\kmax^{3}/M)$ in the uniform probability space over $\Lambda_r(\kvec)$.  

The following lemma will be used repeatedly. In most applications, $c$ will be a small positive integer.
(Throughout the paper, ``$\log$'' denotes the natural logarithm.)

\begin{lemma}
\label{c-parts}
Let $U_1,\ldots, U_c$ be fixed, disjoint $r$-subsets of the set of points $\mathcal{B}$, 
where $r\geq 3$ is a fixed integer and $c = o(M^{1/2})$.
The probability that a uniformly random $Q\in\Lambda_r(\kvec)$ contains the parts 
$\{U_1,\ldots, U_c\}$ is
\[ 
 (1+o(1)) \frac{((r-1)!)^c}{M^{c(r-1)}}.
\]
\end{lemma}

\begin{proof}
Using (\ref{Lambda}), the required probability is
\begin{align*}
\frac{r!^c\, (M/r)_c}{(M)_{rc}} &=  \frac{(r-1)!^c}{M^{(r-1)c}}\,  
           \exp\left( -\sum_{j=0}^{rc-1} \,\log(1 - j/M) + \sum_{i=0}^{c-1}\, \log(1-ri/M) \right)\\
        &= \frac{(r-1)!^c}{M^{(r-1)c}}\,  \exp\left( O\left(\frac{r^2 c^2}{M}\right)\right).
\end{align*}
But $r^2c^2 = o(M)$ by assumption, which completes the proof.
\end{proof}

Let
\[
  N = \max\{ \lceil \log M\rceil,\, \lceil 9(r-1) M_2/M\rceil\}.
\]
Now define $\Lambda_r^+(\kvec)$ to
be the set of partitions $Q\in \Lambda_r(\kvec)$ which satisfy
the following properties:
\begin{enumerate}
\item[(i)]  For each part $U\in Q$ we have $|U\cap B_i|\leq 2$
for $i=1,\ldots, n$.
\item[(ii)] For each part $U\in Q$ there is at most one
$i\in \{ 1,\ldots, n\}$ with $|U\cap B_i|=2$.
\item[(iii)]  
For each pair $(U_1,U_2)$ of distinct parts in $Q$, the 
intersection $e_1\cap e_2$ of the corresponding edges contains at
most 2 vertices.
(It is possible that $e_1\cap e_2$ consists of a loop.)
\item[(iv)] There are at most $N$ parts which contain loops. 
\end{enumerate}

Note in particular that whenever $r\geq 3$,
property (iii) implies that $G(Q)$ has no repeated edges. 

\begin{lemma}
Under the assumptions of Theorem~\ref{main}, we have
\[ \frac{|\Lambda_r^+(\kvec)|}{|\Lambda_r(\kvec)|} = 1 + O(\kmax^{3}/M).
\]
\label{pr-simple}
\end{lemma}

\begin{proof}
Consider $Q\in\Lambda_r(\kvec)$ chosen uniformly at random. 

(i) The expected number of parts in $Q$ which contain
three or more points from the same cell is 
\[ O\left(\frac{M_3 M^{r-3}}{M^{r-1}}\right) = O(\kmax^2/M),\]
using Lemma~\ref{c-parts}.
Hence, the probability that property (i) fails to hold is also $O(\kmax^2/M)$.

(ii)
Similarly, the expected number of parts in $Q$
which contain two loops (where each loop is from a distinct cell) is
\[ O\left(\frac{M_2^2 M^{r-4}}{M^{r-1}}\right) = O(\kmax^2/M).\]

(iii)
Using Lemma~\ref{c-parts},
the expected number of ordered pairs of distinct parts $(U_1,U_2)$ which 
give rise to edges $e_1, e_2$ such that $|e_1\cap e_2| \geq  3$ is 
\[
O\left(\frac{M_2^3\, M^{2(r-3)} + M_2 M_4 M^{2(r-3)}}{M^{2(r-1)}}\right) 
= O(\kmax^3/M).
\]
(Here the first term arises if $e_1\cap e_2$ does not contain a loop
while the second term covers the possibility that $e_1\cap e_2$ contains
a loop. By (i) we can assume that $e_1\cap e_2$ contains at least two distinct
vertices.)

(iv)  Let $\ell=N + 1$.  We bound the
expected number of sets $\{ U_1,\ldots, U_\ell\}$ of $\ell$ parts
which each contain a loop.  Given  $(U_1,\ldots, U_{i-1})$,
there are at most $M_2 M^{r-2}/(2(r-2)!)$ choices for $U_i$.
Hence there are 
\[ O\left(\frac{1}{\ell!}\, \left(\frac{M_2 M^{r-2}}{2(r-2)!}\right)^{\ell}
  \right)
\]
possible sets $\{ U_1,\ldots, U_\ell\}$ of parts which each contain
a loop.  Now 
\[ \ell = O(N) = O(\kmax +  \log M) = o(M^{1/2}),
\]
by definition of $N$.  Hence Lemma~\ref{c-parts} applies, and we conclude
that
the expected number of sets of $\ell = N + 1$ parts which each contain
a loop is
\[ O\left(\frac{1}{\ell!}\, \left(\frac{(r-1)M_2}{2M}\right)^\ell\right) 
  = O\left(\left(\frac{e(r-1)M_2}{2\ell M}\right)^\ell\right)
    = O\left((e/18)^{\log M}\right) 
       = o(1/M),\]
completing the proof.
\end{proof}

In Section~\ref{s:switchings} we will calculate $|\Lambda^+_{r}(\kvec)|$
by analysing switchings which make local changes to a partition to reduce
(or increase) the number of loops by precisely 1.

\section{The switchings}\label{s:switchings}

For a given nonnegative integer $\ell$, let $\mathcal{C}_\ell$
be the set of partitions $Q\in \Lambda_r^+(\kvec)$  with exactly $\ell$ 
parts which contain a loop.
Then partitions in $\mathcal{C}_0$ give rise to hypergraphs in $\Hrk$. 
Now $\mathcal{C}_0$ is nonempty whenever $r$ divides $M$, and we
restrict ourselves to this situation.  Hence 
it follows from Lemma~\ref{pr-simple} that
\begin{equation} \frac{1}{P_r(\kvec)} = 
  \bigl(1 + O(\kmax^{3}/M)\bigr)\,  \sum_{\ell = 0}^{N}\,
   \frac{|\mathcal{C}_\ell|}{|\mathcal{C}_0|}.
\label{a=0}
\end{equation}
We estimate the above sum using a switching designed
to remove loops.

An $\ell$-\emph{switching} in a partition $Q$
is specified by a 4-tuple $(x_1,x_2,y_1,y_2)$ of points
where $x_1$ belongs to the part $U$, and $y_j$ belongs to the part
$W_j$ for $j=1,2$, such that:
\begin{itemize}
\item $U$, $W_1$ and $W_2$ are distinct parts of $Q$, 
\item $y_1$ and $y_2$ belong to distinct cells, and 
\item $U$ contains a loop $\{ x_1, x_2\}$ (so in particular, $x_1$ and $x_2$ belong to
the same cell).
\end{itemize}

The $\ell$-switching maps $Q$ to the partition $Q'$ defined by
\begin{equation}
\label{QQ'-loops}
 Q' = \bigl(Q - \{ U, W_1, W_2\}\bigr) \cup 
                       \{\widehat{U}, \widehat{W}_1, \widehat{W}_2\} 
\end{equation}
where 
\[ \widehat{U} = \bigl( U - \{ x_1,x_2\}\bigr) \cup \{ y_1,y_2\},\quad
  \widehat{W}_1 = \bigl(W_1 - \{ y_1\}\bigr)\cup \{ x_1\},\quad
  \widehat{W}_2 = \bigl(W_2 - \{ y_2\}\bigr) \cup \{ x_2\}.
\] 
This operation is illustrated in Figure~\ref{f:l-switch}.
It is the same operation used by Dudek et al.~\cite{DFRS},
but we use a somewhat different approach when analysing the switching.

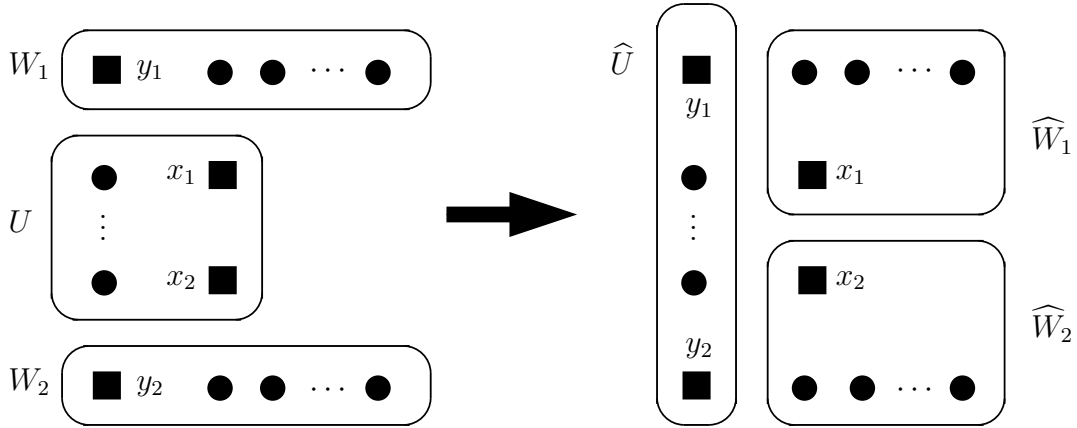
\begin{figure}[ht]
\begin{center}
\unitlength=1cm
\begin{picture}(15,6)(0,0)
\put(0.8,0.0){
    \includegraphics[scale=0.7]{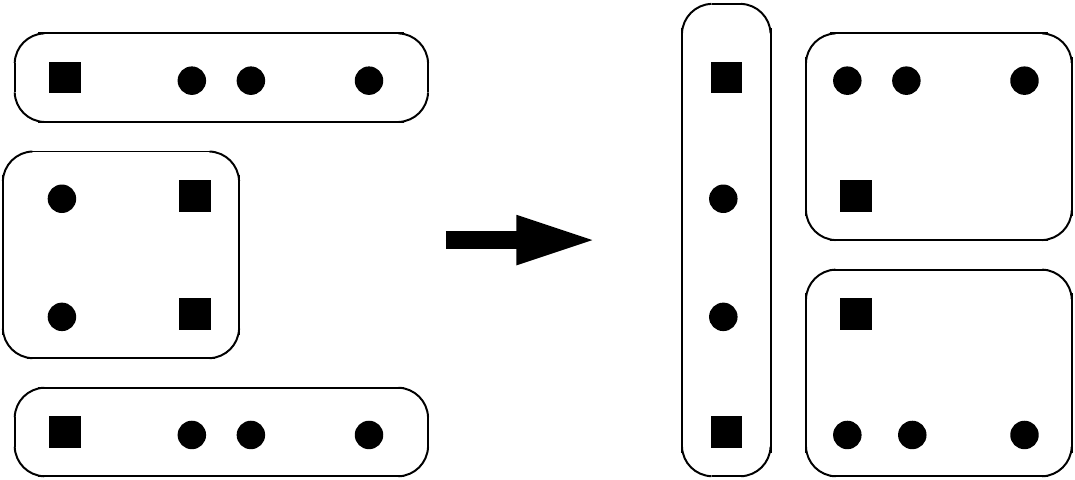}
        }
\put(0.4,4.7){$W_1$}
\put(0.4,2.6){$U$}
\put(1.6,2.5){$\vdots$}
\put(0.4,0.5){$W_2$}
\put(14.0,3.7){$\widehat{W}_1$}
\put(12.2,4.65){$\cdots$}
\put(14.0,1.2){$\widehat{W}_2$}
\put(12.2,0.4){$\cdots$}
\put(8.4,4.7){$\widehat{U}$}
\put(9.45,2.5){$\vdots$}
\put(2.5,3.3){$x_1$}
\put(2.5,1.9){$x_2$}
\put(11.4,3.3){$x_1$}
\put(11.4,1.9){$x_2$}
\put(2.1,4.7){$y_1$}
\put(4.4,4.65){$\cdots$}
\put(2.1,0.5){$y_2$}
\put(4.4,0.4){$\cdots$}
\put(9.4,4.2){$y_1$}
\put(9.4,1.0){$y_2$}
\end{picture}
\caption{An $\ell$-switching}
\label{f:l-switch}
\end{center}
\end{figure}
Let $e$ be the edge of $G(Q)$  corresponding to $U$, and
let $f_j$ be the edge of $G(Q)$ corresponding to $W_j$, for $j=1,2$.
Similarly, let $\widehat{e}$ be the edge of $G(Q')$ corresponding
to $\widehat{U}$, and let $\widehat{f}_j$ be the edge of $G(Q')$
corresponding to $\widehat{W}_j$ for $j=1,2$.

Given $Q\in\mathcal{C}_\ell$, we say that the $\ell$-switching
specified by the 4-tuple of points $(x_1,x_2,y_1,y_2)$ 
is \emph{legal for $Q$}  if the resulting partition $Q'$ belongs to $\mathcal{C}_{\ell-1}$,
and otherwise we say that the switching is \emph{illegal for $Q$}.

\begin{lemma}
With notation as above, if the $\ell$-switching $(x_1,x_2,y_1,y_2)$ is
illegal for $Q$ then at least one of the following conditions must hold:
\begin{itemize}
\item[\emph{(I)}] At least one of $W_1$, $W_2$ contains a loop.
\item[\emph{(II)}] $e$, $f_1$ and $f_2$ are not pairwise disjoint.
\item[\emph{(III)}] Some edge of $G(Q)\setminus \{ e,\, f_1,\, f_2\}$ intersects
both $e$ and  $f_j$, for some $j\in \{1,2\}$.
\end{itemize}
\label{forward-legal}
\end{lemma}

\begin{proof}
Given $Q\in\mathcal{C}_\ell$, suppose that the 4-tuple $(x_1,x_2,y_1,y_2)$ 
specifies
an $\ell$-switching in $Q$ such that the resulting partition $Q'$ 
does not belong to $\mathcal{C}_{\ell-1}$.

It could be that $Q'\in\Lambda_r^+(\kvec)$ but that $Q'$
has strictly more than $\ell-1$ parts which contain a loop.  
Here the $\ell$-switching has (accidently)
introduced at least one new loop.  But this implies that (II) holds,
since we know that $y_1$ and $y_2$ do not belong to the same cell.

Next, suppose that $Q'\in\Lambda_r^+(\kvec)$ but that $Q'$
has at most $\ell-2$ parts which contain a loop.  
This means that the $\ell$-switching has
removed more than one loop. Then property (I) must hold: the point $y_j$
must have been involved in a loop in $W_j$ for some $j\in\{ 1,2\}$.

It remains to consider the case that $Q'\not\in\Lambda_r^+(\kvec)$.
Then at least one of the properties (i)--(iv) used to 
define $\Lambda_r^+(\kvec)$ no
longer holds for $Q'$.  Arguing as above, if (i), (ii) or (iv) fails
then we have introduced at least one loop, or increased the multiplicity of
a vertex in some edge from 2 to at least 3.  This implies that (I) or (II)
holds, using arguments similar to those above.

Finally, suppose that (iii) fails for $Q'$.  Then $G(Q')$ has
a pair of edges which intersect in at least 3 vertices.  
We say that this pair of edges has \emph{large intersection}.
At least one of the new edges $\widehat{e}$, $\widehat{f}_1$, $\widehat{f}_2$ 
must be involved in any such pair, since $Q\in\Lambda_r^+(\kvec)$.

If $\widehat{f}_1$ and $\widehat{f}_2$
have large intersection then $f_1$ and $f_2$ are not 
disjoint, which shows that (II) holds.  Similarly,
if $\widehat{e}$ and $\widehat{f}_j$ have large intersection
for some $j\in \{ 1,2\}$ then $e$ and $f_j$ are not disjoint, and (II) holds.
Now suppose that an edge $e'\in G(Q')\setminus \{ \widehat{e},\, \widehat{f}_1,\,
 \widehat{f}_2\}$ has large intersection with one of the new edges.  
Note that $e'$ is also an edge of $G(Q)\setminus \{ e,f_1,f_2\}$.
\begin{itemize}
\item
If $e'$ has large intersection with $\widehat{f}_j$ for some $j\in \{ 1,2\}$
 then $e'$ must
contain the vertex corresponding to the point $x_j$, or else $e'$ and $f_j$
would have large intersection in $G(Q)$, contradicting the fact that $Q\in\Lambda_r^+(\kvec)$.
Furthermore, $e'\cap \widehat{f}_j$ contains at least one other vertex, 
corresponding to a point in $\widehat{W}_j\setminus \{ x_j\} = W_j \setminus \{ y_j\}$.  
Hence $e'$
intersects both $e$ and $f_j$ in $G(Q)$, showing that (III) holds.
\item 
If $e'$ has large intersection with $\widehat{e}$ then $e'$ must
contain the vertex corresponding to $y_j$ for some $j\in \{ 1,2\}$ (perhaps both),
otherwise $e'$ and $e$ would have large intersection in $G(Q)$, a contradiction.
Even if $e'$ contains both of these vertices, it must still contain a vertex
corresponding to a point in $\widehat{U} \setminus \{ y_1,y_2\} = U \setminus \{ x_1,x_2\}$.
Hence $e'$ intersects both $f_j$ and $e$ in $G(Q)$ for some $j\in \{ 1,2\}$, which
again proves that (III) holds.
\end{itemize}
This completes the proof.  
\end{proof}

A \emph{reverse $\ell$-switching} in a given partition $Q'$
is the reverse of an $\ell$-switching.
It is described by a 4-tuple $(x_1,x_2,y_1,y_2)$ of points,
where $\widehat{W}_j$
is the part of $Q'$ containing $x_j$, for $j=1,2$, and $y_1, y_2$
are distinct points in the part $\widehat{U}$ of $Q'$, such that
\begin{itemize}
\item $\widehat{U}$, $\widehat{W}_1$ and $\widehat{W}_2$ are distinct
parts of $Q'$,
\item $x_1$ and $x_2$ belong to the same cell, and
\item $y_1$ and $y_2$ belong to distinct cells.
\end{itemize}
This reverse $\ell$-switching acting on $Q'$ produces the partition $Q$
defined by (\ref{QQ'-loops}),   
as depicted in Figure~\ref{f:l-switch}
by following the arrow in reverse.  
Given $Q'\in\mathcal{C}_{\ell-1}$, we say that the
reverse $\ell$-switching specified by $(x_1,x_2,y_1,y_2)$
is \emph{legal for $Q'$} if the resulting partition $Q$ belongs
to $\mathcal{C}_\ell$, and otherwise we say that the switching is
\emph{illegal for $Q'$}.
For completeness we give the full proof of the following, though
it is very similar to the proof of Lemma~\ref{forward-legal}.

\begin{lemma}
\label{legal-reverse}
With notation as above, if the reverse $\ell$-switching
specified by $(x_1,x_2,y_1,y_2)$ is illegal for $Q'\in\mathcal{C}_{\ell-1}$
then at least one of the following conditions must hold:
\begin{itemize}
\item[\emph{(I${}'$)}]
At least one of $\widehat{U}$, $\widehat{W}_1$, $\widehat{W}_2$ contains a loop.
\item[\emph{(II${}'$)}]  
$\widehat{e}\cap\widehat{f}_j\neq \emptyset$ for some $j\in \{ 1,2\}$.
\item[\emph{(III${}'$)}] Some edge of $G(Q')\setminus \{ \widehat{e}, \widehat{f}_1, \widehat{f}_2\}$
intersects both $\widehat{e}$ and $\widehat{f}_j$ for some $j\in\{ 1,2\}$.
\end{itemize}
\end{lemma}

\begin{proof}
Fix $Q'\in \mathcal{C}_{\ell-1}$ and let $(x_1,x_2,y_1,y_2)$ describe
an reverse $\ell$-switching such that the resulting partition $Q$
does not belong to $\mathcal{C}_\ell$.

If $Q\in\Lambda_r^+(\kvec)$ but $Q$ has more than $\ell$ parts
which contain loops then an extra loop has been unintentionally
introduced. In this case, either $\widehat{W}_j\setminus \{ x_j\}$
contains a point from the same cell as $y_j$, or $\widehat{U}\setminus \{ y_1,y_2\}$
contains a point from the same cell as $x_j$, for some $j\in\{ 1,2\}$.
In either case we have $\widehat{e}\cap \widehat{f}_j\neq\emptyset$,
so (II${}'$) holds.
Next, suppose that $Q\in\Lambda_r^+(\kvec)$ but that $Q$ has at
most $\ell-1$ parts which contain a loop.  Then the reverse
switching has removed at least one loop, which implies that
(I${}'$) holds.

Now suppose that $Q\not\in\Lambda_r^+(\kvec)$.  Then one of
the properties (i)--(iv) fail for $Q$. If (i), (ii) or (iv) fail then
arguing as above we see that (I${}'$) or (II${}'$) holds.
Now suppose that (iii) fails.  Then
some edge of $G(Q)$ has large intersection with one of $e, f_1, f_2$
(recalling that terminology from the proof of Lemma~\ref{forward-legal}).
Now $f_1$ and $f_2$ cannot have large intersection, since their intersection
is contained in the intersection of $\widehat{f}_1$ and $\widehat{f}_2$,
and $Q'\in\Lambda_r^+(\kvec)$.  If $e$ and $f_j$ have large intersection for
some $j\in\{ 1,2\}$ then either this intersection contains the vertex corresponding
to $x_j$ (and hence $\widehat{W}_j$ contains a loop), or 
the intersection contains the vertex corresponding to $y_j$
(and hence $\widehat{U}$ contains a loop), or
$\widehat{e}\cap\widehat{f}_j\neq \emptyset$. 
Again (I${}'$) or (II${}'$) hold.

Finally, suppose that the large intersection involves an edge $e'\in G(Q)\setminus \{ e,f_1,f_2\}$.
Then $e'$ also belongs to $G(Q')\setminus \{ \widehat{e},\, \widehat{f}_1,\, \widehat{f}_2\}$.
If $e'$ has large intersection with $e$ in $G(Q)$ then $e'$ contains the
vertex corresponding to the point $x_j$, for some $j\in\{1,2\}$ (or else $e'$ and $\widehat{e}$ have
large overlap in $G(Q')$, a contradiction), and $e'$ contains at
least one vertex corresponding to a point of $U\setminus \{ x_1,x_2\} = \widehat{U}\setminus \{ y_1,y_2\}$.
Therefore $e'$ overlaps both $\widehat{e}$ and $\widehat{f}_j$, so (III${}'$) holds.
Similarly, if $e'$ has large intersection with $\widehat{f}_j$ for some $j\in\{1,2\}$
then $e'$ contains the vertex corresponding to $y_j$ (or else $e'\cap \widehat{f}_j$ is large
in $G(Q')$, a contradiction), and $e'$ contains at least one vertex corresponding to
a point in $W_j \setminus \{ y_j\} = \widehat{W}_j\setminus \{ x_j\}$.  Again,
$e'$ overlaps both $\widehat{e}$ and $\widehat{f}_j$, proving that (III${}'$) holds, as required.
\end{proof}

Next we analyse these switchings to find a relationship between the 
sizes of $\mathcal{C}_\ell$ and $\mathcal{C}_{\ell-1}$.

\begin{lemma}
Assume that the conditions of Theorem~\ref{main} hold and let 
$\ell'$ be the first value of
$\ell\leq N$ such that $C_{\ell}=\emptyset$, or $\ell'=N+1$ if no
such value exists.  Then 
\[
  |\mathcal{C}_\ell| = |\mathcal{C}_{\ell-1}|\, 
    \frac{(r-1)M_2}{2\ell M}\, \left( 1 + O\left( 
    \frac{\kmax^{3} + \ell\, \kmax}{M_2} \right)\right)
\]
uniformly for $1\leq \ell <  \ell'$. 
\label{ratio} 
\end{lemma}

\begin{proof}
Fix $\ell\in \{ 1,\ldots, \ell'-1\}$ 
and let $Q\in\mathcal{C}_\ell$ be given.
Define the set $\mathcal{S}$ of all 4-tuples $(x_1,x_2,y_1,y_2)$
of distinct points such that 
\begin{itemize}
\item $y_1$ and $y_2$ belong to distinct cells,
\item $\{ x_1, x_2\}$ is a loop in $U$ and $y_j\in W_j$ for $j=1,2$,
for some distinct parts $U, W_1, W_2\in Q$, and 
\item  neither $W_1$ nor $W_2$ contain a loop. 
\end{itemize}
Note that $\mathcal{S}$ contains every 4-tuple which defines a 
legal $\ell$-switching from $Q$, so $|\mathcal{S}|$ is an upper bound
for the number of legal $\ell$-switchings which can be performed in $Q$.

There are precisely $2\ell$ ways to choose a pair of points $(x_1,x_2)$ 
which form a loop in some part $U$,
using properties (i) and (ii) of the definition of $\Lambda_r^+(\kvec)$.
For an easy upper bound, there are at most $M^2$ ways to select $(y_1,y_2)$ with
the required properties, giving $|\mathcal{S}|\leq 2\ell M^2$.
In fact
\begin{equation}
 |\mathcal{S}| = 2\ell\, M^2 \left(1 + O\left(\frac{\kmax + \ell}{M}\right)\right),
\end{equation}
since there are precisely $M-r\ell$ ways to select a point $y_1$ which belongs to
some part $W_1$ which does not contain a loop, and then there are
$M - r(\ell+1) + O(\kmax) = M + O(\kmax + \ell)$ ways to select a point $y_2$ which lies in a part $W_2$
which contains no loops and which is
distinct from $W_1$, such that $y_1$ and $y_2$ not in the same
cell.

We now find an upper bound for the number of 4-tuples in $\mathcal{S}$ 
which give rise to illegal $\ell$-switchings,  and subtract this value from 
$|\mathcal{S}|$.
By Lemma~\ref{forward-legal} it suffices to find an upper bound
for the number of 4-tuples
in $\mathcal{S}$ which satisfy one of Conditions (I), (II), (III).
First note that no 4-tuple in $\mathcal{S}$ satisfies Condition (I), 
by definition of $\mathcal{S}$.

If Condition (II) holds then $f_1\cap f_2\neq \emptyset$ or
$e\cap f_j\neq \emptyset$ for some $j\in\{1,2\}$.
This occurs for at most $O(\ell\kmax M)$ 4-tuples in $\mathcal{S}$.

If Condition (III) holds then some edge $e'$ of $G(Q)\setminus \{ e, f_1, f_2\}$ 
intersects two of $e$, $f_1$ and $f_2$.  There are $O(\ell\kmax^2 M)$ choices
of 4-tuples in $\mathcal{S}$ which satisfy this condition.

Combining these contributions, we find that there are
\begin{equation}
\label{l-for}
 2\ell M^2\left( 1 + O\left(\frac{\kmax^{2} + \ell}{M}\right)\right)
\end{equation}
4-tuples $(x_1,x_2,y_1,y_2)$ which give a legal $\ell$-switching
from $Q$.

Next, suppose that $Q'\in \mathcal{C}_{\ell-1}$ (and note that $\mathcal{C}_{\ell-1}$
is nonempty, by definition of $\ell'$).  
Let $\mathcal{S}'$ be the set of all 4-tuples $(x_1,x_2,y_1,y_2)$ of
distinct points such that 
\begin{itemize}
\item $x_1$ and $x_2$ belong to the same cell, 
\item 
$x_j\in \widehat{W}_j$ for $j=1,2$ and 
and $y_1,y_2\in \widehat{U}$, 
for some distinct parts $\widehat{U}$, $\widehat{W}_1$, $\widehat{W}_2$ of $Q'$, and
\item
$\widehat{U}$ does not contain a loop (so in particular, $y_1$ and $y_2$ belong to distinct cells).  
\end{itemize}
Again, $\mathcal{S}'$ contains every 4-tuple which describes a 
legal reverse $\ell$-switching from $Q'$, so the number of legal
reverse $\ell$-switchings which may be performed in $Q'$ is at most
$|\mathcal{S}'|$.
There are $M_2$ choices for $(x_1,x_2)$,
and each such choice determines two distinct parts $\widehat{W}_1$, $\widehat{W}_2$
unless $\{x_1,x_2\}$ is a loop in some part of $Q'$.
Using properties (i) and (ii) of the definition of $\Lambda_r^+(\kvec)$,
there are exactly $2(\ell-1)$ choices of $(x_1,x_2)$ such that
$\{ x_1,x_2\}$ is a loop in $Q'$.
Next, there are precisely $M-r(\ell-1)$ choices for $y_1$ belonging to some part
$\widehat{U}$ which does not contain a loop, and then there are $r-1$ choices for
$y_2\in \widehat{U}\setminus\{ y_1\}$.  For a lower bound, there are
at least $(r-1)(M-r(\ell+1))$ choices for $(y_1,y_2)$ which
ensure that $\widehat{U}$ contains no loop and is distinct from both 
$\widehat{W}_1$ and $\widehat{W}_2$.  Therefore
\[ (r-1)\left(M - r(\ell+1)\right)\, \left(M_2 - 2(\ell-1)\right)\leq |\mathcal{S}'| \leq
    (r-1)\left(M -  r(\ell-1)\right)\, M_2, 
\]
which implies that $|\mathcal{S}'| = (r-1)M\, M_2\, (1 + O(\ell/M + \ell/M_2))$.

Now we must find an upper bound for the number of 4-tuples in $\mathcal{S}'$ 
which give an illegal reverse $\ell$-switching in $Q$, and subtract this number 
from $|\mathcal{S}'|$. 
By Lemma~\ref{legal-reverse} it suffices to find upper bounds for
the number of elements
of $\mathcal{S}'$ which satisfy (at least) one of conditions (I${}'$), (II${}'$) or (III${}'$).
If Condition (I${}'$) holds then $\widehat{W}_j$ contains a loop
for some $j\in\{1,2\}$, which is true for $O(\ell \kmax M)$ 4-tuples in 
$\mathcal{S}'$.
(Recall that $\widehat{U}$ has no loop, by definition of $\mathcal{S}'$.)
Condition (II${}'$) holds if $\widehat{e}\cap\widehat{f_j}$ is nonempty for
some $j\in \{ 1,2\}$.  This occurs for at most $O(\kmax M_2)$ 4-tuples in $\mathcal{S}'$.
Next, suppose that Condition (III${}'$) holds.  
Then there exists an edge $e'\in G(Q')\setminus \{ \widehat{e},\, \widehat{f}_1,\,
\widehat{f}_2\}$ which intersects both $\widehat{e}$ and $\widehat{f}_j$ for
some $j\in \{ 1,2\}$.  The number of 4-tuples in $\mathcal{S}'$ which satisfy this 
condition is $O(\kmax^2 M_2)$.

Putting these contributions together, the number of 4-tuples in $\mathcal{S}'$
which give a legal reverse $\ell$-switchings from $Q'$ is
\begin{equation}
\label{l-rev} 
  (r-1)M M_2\left(1 + O\left(\frac{\kmax^2}{M} + \frac{\ell\kmax}{M_2}\right) \right) =
  (r-1)M M_2\left(1 + O\left(\frac{\kmax^3 + \ell\kmax}{M_2}\right) \right),
\end{equation}
since $1/M \leq \kmax/M_2$.
Combining (\ref{l-for}) and (\ref{l-rev}) completes the proof. 
\end{proof}

The following summation lemma from~\cite{GMW} will be needed,
and for completeness we state it here. (The statement has
been adapted slightly from that given in~\cite{GMW}, without affecting 
the proof given there.)

\begin{lemma}[{\cite[Corollary 4.5]{GMW}}]\label{sumcor2}
Let $N\geq 2$ be an integer and, for $1\leq i\leq N$, let real
numbers $A(i)$, $C(i)$ be given such that $A(i)\geq 0$ and
$A(i)-(i-1)C(i) \ge 0$.
Define 
$A_1 = \min_{i=1}^N A(i)$, $A_2 = \max_{i=1}^N A(i)$,
$C_1 = \min_{i=1}^N C(i)$ and $C_2=\max_{i=1}^N C(i)$.
Suppose that there exists a real number $\hat{c}$ with 
$0<\hat{c} < \tfrac{1}{3}$ such that 
$\max\{ A_2/N,\, |C_1|, \, |C_2|\} \leq \hat{c}$.
Define $n_0,\ldots ,n_N$ by $n_0=1$ and
\[ n_i  = \frac{1}{i}\bigl(A(i)-(i-1)C(i)\bigr)\, n_{i-1}\]
for $1\leq i\leq N$.  Then
\[ \varSigma_1 \leq \sum_{i=0}^N n_i\leq \varSigma_2, \]
where
\begin{align*}
 \varSigma_1 &= \exp\bigl( A_1 - \tfrac{1}{2} A_1 C_2 \bigr)
               - (2e\hat{c})^N,\\
 \varSigma_2 &= \exp\bigl( A_2 - \tfrac{1}{2} A_2 C_1 +
              \tfrac12 A_2 C_1^2 \bigr) + (2e\hat{c})^N.
     \quad\qedsymbol
\end{align*}
\end{lemma}

This summation lemma will now be applied.

\begin{lemma}
Under the conditions of Theorem~\ref{main} we have
\[ \sum_{\ell = 0}^{N} |\mathcal{C}_\ell| = |\mathcal{C}_0| \,
  \exp\left( \frac{(r-1)M_2}{2M} + 
    O\left(\frac{\kmax^{3}}{M}\right)\right).\]
\label{l-switch}
\end{lemma}

\begin{proof}
Let $\ell'$ be as defined in Lemma~\ref{ratio}.
By (\ref{l-for}), any $Q\in\mathcal{C}_\ell$ can be converted
to some $Q'\in\mathcal{C}_{\ell-1}$ using an $\ell$-switching.  Hence
$\mathcal{C}_\ell=\emptyset$ for $\ell'\leq \ell\leq N$. In particular,
the lemma holds if $\mathcal{C}_0=\emptyset$, so we assume that $\ell'\geq 1$.

By Lemma~\ref{ratio}, there exists some  uniformly bounded function
$\beta_\ell$ such that
\begin{equation}\label{lrat}
\frac{|\mathcal{C}_\ell|}{|\mathcal{C}_0|} = \frac{1}{\ell}\,
  \frac{|\mathcal{C}_{\ell-1}|}{|\mathcal{C}_0|}\,  \bigl( A(\ell) - (\ell-1)C(\ell)\bigr)
\end{equation}
for $\ell = 1,\ldots, N$, where
\[
  A(\ell) = \frac{(r-1)M_2-\beta_\ell\, \kmax^{3}}{2M},
  \quad C(\ell) = \frac{\beta_\ell\, \kmax}{2M}
\]
for $1\le\ell<\ell'$,  and $A(\ell) = C(\ell) = 0$ for $\ell'\leq \ell \leq N$.

Now we apply Lemma~\ref{sumcor2}.
It is clear that $A(\ell) - (\ell-1)C(\ell)\geq 0$, from (\ref{lrat})
if $1\leq \ell < \ell'$, or by definition if $\ell'\leq \ell \leq N$.
If $\beta_\ell \geq 0$ then $A(\ell)\geq A(\ell)-(\ell-1)C(\ell) \geq 0$,
while if $\beta_\ell < 0$ then $A(\ell)$ is nonnegative by definition.
Next, define $A_1, A_2, C_1, C_2$ to be the minimum and maximum
of $A(\ell)$ and $C(\ell)$ over $1\leq \ell\leq N$, as in
Lemma~\ref{sumcor2},  
and set $\hat c=\frac{1}{16}$.   Since $A_2 = (r-1)M_2/(2M) + o(1)$
and $C_1,C_2 = o(1)$, we have that $\max\{ A_2/N,\, |C_1|,\, |C_2|\}\leq \hat{c}$
for $M$ sufficiently large, by definition of $N$.  
Lemma~\ref{sumcor2} applies and gives an upper bound
\[ \sum_{\ell = 0}^{N} \frac{|\mathcal{C}_\ell|}
   {|\mathcal{C}_0|} 
 \leq  \exp\left( \frac{(r-1)M_2}{2M} + 
    O\left(\frac{\kmax^{3}}{M}\right)\right) 
       + O\bigl( (e/8)^{N}\bigr).\]
Now $(e/8)^{N}\leq (e/8)^{\log M} \leq M^{-1}$, 
which leads to
\begin{equation}
\label{upper} \sum_{\ell = 0}^{N} \frac{|\mathcal{C}_\ell|}
   {|\mathcal{C}_0|} 
 \leq  \exp\left( \frac{(r-1)M_2}{2M} + 
    O\left(\frac{\kmax^{3}}{M}\right)\right). 
\end{equation} 
If $\ell' = N + 1$ then the lower bound given by Lemma~\ref{sumcor2}
is the same as the upper bound (\ref{upper}), within the stated error term, 
establishing the result in this case.

Finally suppose that $1\leq \ell'\leq N$.  Then (\ref{l-rev})
shows that 
\[ M_2 = O(\kmax^3 + \ell' \kmax) = o(M + M^{1/3}\log M) = o(M).\]
If $\ell'=1$ then
$M_2 = O(\kmax^{3})$ and hence $(r-1)M_2/(2M) = O\bigl(\kmax^3/M\bigr)$,
so in this case the trivial lower bound of 1 matches the upper bound (\ref{upper}), 
within the stated error term.  
If $2\leq \ell'\leq N$ then using (\ref{lrat}) with $\ell=1$, we obtain
\[ \sum_{\ell = 0}^{N} \frac{|\mathcal{C}_\ell|}{|\mathcal{C}_0|} 
  \geq 1 + \frac{|\mathcal{C}_1|}{|\mathcal{C}_0|} = 
    1 + A(1) = 1 + \frac{(r-1)M_2}{2M} + O\bigl(\kmax^3/M\bigr).
\] 
Since here $M_2 = o(M)$, this expression matches the upper bound (\ref{upper}), within the
stated error term.  This completes the proof.
\end{proof}

Theorem~\ref{main} now follows immediately, by 
combining (\ref{pr-equation}), (\ref{a=0}) and Lemma~\ref{l-switch}.

\end{document}